\documentclass{elsarticle}
\usepackage{amsmath}
\usepackage{amsthm}
\usepackage{amsfonts,mathrsfs}
\usepackage{amssymb}
\usepackage{graphicx}
\usepackage{enumitem}
\usepackage{color}
\usepackage{verbatim}
\usepackage{url}
\usepackage{tikz}
\usepackage{hyperref}
\usepackage{xpatch} % Modifikation von BibLaTeX möglich
\theoremstyle{plain}
\newtheorem{theorem}{Theorem}[section]
\newtheorem{lemma}[theorem]{Lemma}
\newtheorem{example}[theorem]{Example}
\newtheorem{corollary}[theorem]{Corollary}
\newtheorem{proposition}[theorem]{Proposition}

\theoremstyle{definition}
\newtheorem{definition}[theorem]{Definition}

\newtheoremstyle{TheoremNum}
{\topsep}{\topsep}              %%% space between body and thm
{\itshape}                      %%% Thm body font
{}                              %%% Indent amount (empty = no indent)
{\bfseries}                     %%% Thm head font
{.}                             %%% Punctuation after thm head
{ }                             %%% Space after thm head
{\thmname{#1}\thmnote{ \bfseries #3}}%%% Thm head spec
\newtheorem*{remark}{Remark}
\newcommand{\abs}[1]{\lvert#1\rvert}
\newcommand{\F}{\mathbb F}

\newcommand{\Z}{\mathbb Z}
\newcommand{\PG}{\mathrm{PG}}
\newcommand{\cS}{\mathcal{S}}
\newcommand{\cP}{\mathcal{P}}
\newcommand{\cL}{\mathcal{L}}

\begin{document}	
	
	\begin{frontmatter}
		
		\title{On the construction of large local arcs}
		
		% First Author
		\author[a]{Ferdinand Ihringer}
		\ead{ihringer@sustech.edu.cn}
		
		% Second Author (Marked as corresponding author)
		\author[b]{Yue Zhou\corref{cor1}}
		\ead{yue.zhou.ovgu@gmail.com}
		
		% Corresponding author text
		\cortext[cor1]{Corresponding author}
		
		% Affiliations (Lowercase letters 'a' and 'b' map to the authors)
		\address[a]{Dept.\ of Mathematics, Southern University of Science and Technology, Shenzhen, Guangdong, China}
		\address[b]{College of Science, National University of Defense Technology, 410073 Changsha, China}
		
		\begin{abstract}
			Motivated by the construction of optimal locally repairable codes, we introduce the new finite geometric concept of a \emph{local arc} which is defined as a collection $\mathcal{S}$ of disjoint point sets $S_{i}$ in $\mathrm{PG}(2,q)$ such that $S_{i} \cup S_{j}$ is an arc for any $S_{i}, S_{j} \in \mathcal{S}$. We focus on the upper and lower bounds on the sizes of maximum $k$-uniform local arcs. For $q=p^m$ with $p$ prime, we construct $k$-uniform local arcs in $\mathrm{PG}(2,q)$ of size $\Omega(q^{d})$ where $d$ is between $1.1167$ and $1.25$ depending only on $m$. For $k=4$, this implies the existence of optimal locally repairable codes (LRCs) with minimum distance 6, locality 3, and disjoint repair groups, whose length is superlinear in $q$--a significant improvement over the previously known $O(q)$ constructions for such LRCs.
		\end{abstract}
		
		% Optional: Keywords
		
	\end{frontmatter}
	
	\section{Introduction}\label{sec:intro}
	Over the past few decades, the notion of locality in coding theory has attracted considerable attention. Traditionally, error-correcting codes rely on examining the entire message to detect corruptions or correct a single error. Locality fundamentally alters this paradigm: it enables the detection of corruptions (as in Locally Testable Codes) or the recovery of data (as in Locally Repairable Codes) by accessing only a small fraction of the codeword.
	
	The study of locality in coding theory has given rise to numerous intriguing combinatorial problems, including, most prominently, the construction of high-dimensional expanders for locally testable codes by Dinur, Evra, Livne, Lubotzky and Mozes in \cite{DELLM}. In this paper, we introduce a new finite geometric concept, termed \emph{local arc}, that is closely related to locally recoverable codes.
	
	Let $\cS$ be a collection of subsets of points in $\PG(2,q)$, the classical projective plane of order $q$.
	We call $\cS$ a \emph{local arc} if
	\begin{itemize}
		\item $S_i\cap S_j=\emptyset$, and
		\item $S_i \cup S_j$ is an arc, i.e.\  no three points of $S_i\cup S_j$ is on a common line,
	\end{itemize}
	for any distinct $S_i,S_j\in \cS$.	We say that $\cS$ is a \emph{$k$-uniform local arc} if $|S_i| = k$ for all $S_i \in \cS$.
	We say that $|\cS|$ is the \emph{size of the local arc $\cS$}.
	
	A trivial construction of $k$-uniform local arcs is a partition of an arc consisting of $k\ell$ points. Thus, there always exist $k$-uniform local arcs of size $\lfloor \frac{q+1}{k} \rfloor$ for $q$ odd and $\lfloor \frac{q+2}{k} \rfloor$ for $q$ even due to the existence of ovals and hyperovals, respectively.
	
	In this note, we are interested in constructions of large $k$-uniform local arcs and the upper bounds of their sizes.
	One of our motivations is from locally repairable codes (LRC). 
	
	An error-correcting code $C$ is said to have \emph{locality} $r$ at $i$-th symbol if $c_i$ of every codeword $c = (c_1, \dots, c_n)\in C$ can be reconstructed by a functional dependence on at most $r$ other symbols.
	If each symbol	of $C$ has locality $r$, $C$ is called a \emph{locally repairable code} (LRC) with locality $r$. A $q$-ary $(n,K,d)$-code $C$ with locality $r$ is called \emph{optimal} or \emph{Singleton-optimal}, if the upper bound in $d\leq n-\log_{q}K-\lceil \log_{q}K/r\rceil +2$ is achieved; see \cite{TB2014} for a proof of the inequality using graph theory. In particular, the length of them can be much bigger than $q+1$ which is quite different from the classical $q$-ary MDS codes where it is conjectured that the code length is upper bounded by $q+1$ (or $q+2$ for some special cases). The maximal length problem of Singleton-optimal linear LRCs and their constructions have become the most important research problems on LRCs in the past 10 years; see \cite{FTFC2024,J2019,GHSY2012,TB2014,XY2022} and the references therein for details.
	
	An $[n,k,d]$-linear code with locality $r$ is often denoted as an $(n,k,d; r)$-LRC. In \cite{FTFC2024}, it is proved that a Singleton-optimal $(n, k, 6; 3)$-LRC with disjoint repair groups is equivalent to a $4$-uniform local arc. More explicitly, if a $4$-uniform local arc $\cS$ has $b=|\cS|$ members, then the corresponding code has parameters
	\[
		[4b,3b-3,6]_q,
	\]
	locality $3$, and $b$ pairwise disjoint repair groups. Each member $S\in\cS$ supplies one repair group of four coordinates, so each coordinate in that group is recoverable from the other three. In the correspondence of \cite{FTFC2024}, the requirement that the union of any two members of $\cS$ is an arc is precisely the geometric compatibility condition used to obtain minimum distance $6$. Thus increasing $|\cS|$ directly increases the code length $n=4|\cS|$.

	Different families of $(q+1, k, 6; 3)$-LRCs and sporadic examples for small $q$ and small length $n$ are obtained in \cite{FTFC2024,XL2026}. 

	Our main result is as follows:

	\begin{theorem}\label{thm:main}
	Fix an integer $k\geq 2$.
	\begin{enumerate}[label=(\roman*)]
		\item For every $k$-uniform local arc $\cS$ in $\PG(2, q)$, 
		$$|\cS|\leq \left( \frac{\sqrt{2}}{k\sqrt{k-1}} + o(1)\right) q^{1.5}$$
		as $q\rightarrow \infty$ with fixed $k$. 
		\item For fixed $m\geq 1$, there exist $k$-uniform local arcs $\cS$ in $\PG(2,q)$ with
		$$|\cS|=\Omega(q^{d_m}),$$
		where $q = p^m$, $p$ prime, and
		\begin{align}
			d_m= 
			\begin{cases}
%				1.2334, & m=1;\\
				1.1167+0.1167/m, & 2\nmid m;\\
				1.1167, & m=4;\\
				1.25-0.2666/m, & 4\nmid m, 2\mid m;\\
				1.25-0.7666/m, & 4\mid m \text{ and } m>4,
			\end{cases}	 \label{eq:localarc}
		\end{align}
		as $p\rightarrow \infty$ through the odd primes.
	\end{enumerate}
% 	 Furthermore, if $q$ is prime then there exists a $k$-uniform local
% 	 arc of size at least $(1+o_k(1)) p^{1.233}$.
	\end{theorem}

	The four lines in \eqref{eq:localarc} reflect the prime-field construction plus odd-degree extensions, the exceptional case $m=4$, and the two parity classes of larger even extension degrees. Every displayed exponent is strictly larger than $1$, so every fixed extension degree gives a superlinear improvement over the previously known $O(q)$ constructions. The prime-field exponent is $1.2334$. For odd $m>1$ the exponent tends to $1.1167$ as $m$ grows, whereas the two larger even-degree families have exponents tending to $1.25$ and hence give the strongest asymptotic improvements.

	Our upper bound asymptotically matches that given in \cite[Theorem~5]{FTFC2024} for the case for $k=4$. The construction for the lower bound is partially built upon a recent result by Hunter, Pohoata, Verstra\"ete, and Zhang \cite{HPVZ2026},	who constructed matchings in the incidence graph of $\PG(2, q)$ of (up to a constant factor) the same size as ours for odd $m$.
	For the lower bound, our result is a significant improvement
	for the previously investigated case of $k=4$ as we immediately obtain
	the following as a consequence of \cite[Theorem 6]{FTFC2024}.

	\begin{corollary}
	For fixed $m\geq 1$, as $p\rightarrow \infty$ through the odd primes and $q=p^m$, there exist $q$-ary Singleton-optimal locally repairable codes with locality $3$, disjoint repair groups, and parameters
	\[
		[n,\kappa,6]_q=[4b,3b-3,6]_q,
		\qquad
		b=\Omega(q^{d_m}),
	\]
	where $d_m$ is defined in \eqref{eq:localarc}.
	\end{corollary}

	To our knowledge, all previously known constructions for such codes
	only have size $O(q)$; see \cite{FTFC2024,XL2026}.

	The remainder of this paper is organized as follows: In Section \ref{sec:upper}, we prove the upper bounds in Theorem \ref{thm:main}, derive several properties of local arcs and introduce another new concept in finite geometry named quasiarc. In Section \ref{sec:cons}, we establish the constructions of $k$-uniform local arcs in Theorem \ref{thm:main} (ii). Section \ref{sec:comp} presents computational results on the sizes of maximum $k$-uniform local arcs for some small $k$ and $q$. In the last section, we present some remarks and questions for future research.
	
	\section{Upper bounds} \label{sec:upper}
	In this section, we prove several properties of local arcs and derive upper bounds on their sizes. For a $k$-uniform local arc $\cS$, we can always derive the associated set of secant line sets as a collection of $\binom{k}{2}$-subsets of lines as follows:
	\[
		\left\{
			\{P_iP_j: 1\leq i<j\leq k\}: S=\{P_1,\dots,P_k\}\in \cS
		\right\}.
	\]
	
	The following result follows directly from the definition of local arcs.
	\begin{lemma}\label{le:k->k-1}
		Let $\cS$ be a $k$-uniform local arc in $\PG(2,q)$ and $\cL$ the associated set of secant line sets. For each $S\in \cS$, remove an arbitrary point $P\in S$ to get $S':=S\setminus \{P\}$ and set $\cS'=\{S':S\in \cS\}$. 
		\begin{itemize}
			\item If $k>2$, then $\cS'$ is a $(k-1)$-uniform local arc;
			\item If $k=2$, then $\{(Q, \ell): S'=\{Q\}\in\cS, L=\{\ell\}\in \cL \}$ is an induced matching in the incidence graph of $\PG(2,q)$.
		\end{itemize}
	\end{lemma}
	Here, the incidence graph of $\PG(2,q)$ is the bipartite graph whose left/right vertex set is the set of points/lines of $\PG(2,q)$, and where a left vertex is adjacent to a right one if and only if the corresponding point and line are incident.
	
	It is well-known that the expander mixing-lemma shows that the size of an induced matching in the incidence graph of $\PG(2,q)$ is upper bounded by
	\begin{equation}\label{eq:trivial_bnd}
		q^{3/2}+1.	
	\end{equation}
	This follows directly from Theorem 3.1.1 in Haemers' PhD thesis \cite{HaemersPhD}
	(with $v=b=q^2+q+1$, $k=r=q+1$, $r_1=k_1=1$, $\sigma_2 = \sqrt{q}$, and $v_1=b_1$ being the size of the matching).
	Motivated by Szemer\'edi-Trotter type incidence theory, the bound has been explicitly stated by Vinh in \cite{Vinh2011}.

	By Lemma \ref{le:k->k-1}, a $k$-uniform local arc in $\PG(2,q)$ has at most $q^{3/2}+1$ members.
	
	In \cite[Theorem 5]{FTFC2024}, the authors use a coding theoretic argument to show
	that a $4$-uniform local arc has size at most
	\begin{equation}\label{eq:theirbnd}
	  |\cS|
	  \leq
	  \left\lfloor
	  \frac{7q+3+\sqrt{24q^3+q^2-6q-63}}{24}
	  \right\rfloor
	  =
	  \left(\frac{\sqrt{2}}{4\sqrt{3}}+o(1)\right)q^{3/2}.
	\end{equation}
	Here we prove a general bound on $k$-uniform local arcs using the well-known expander-mixing lemma for bipartite graphs.
	See \cite[Theorem 3.1.1]{HaemersPhD} (in the language of design theory) or \cite{AC1988}.

	\begin{lemma}[Expander-Mixing Lemma for Biregular Graphs]\label{EML}
	 Let $\Gamma$ be a $d$-regular bipartite graph on $2N$ vertices with parts $X_1$ and $X_2$.
	 Let $Y_i \subseteq X_i$ with $y_i = |Y_i|$.
	 Then
	 \[
	  \left| N E(Y_1, Y_2) - d y_1 y_2 \right| \leq \lambda_2 \sqrt{y_1 y_2 (N-y_1)(N-y_2)}.
	 \]
	 Here $E(Y_1, Y_2)$ denotes the number of edges between $Y_1$ and $Y_2$,
	 and $\lambda_2$ is the second largest eigenvalue of the adjacency matrix of $\Gamma$.
	\end{lemma}

	With this, we can show an upper bound on the size of $k$-uniform local arcs.

	\begin{theorem}\label{thm:upper}
	 A $k$-uniform local arc has size at most
	 \begin{equation}\label{eq:ourbnd}
	  \left\lfloor
	  \frac{
	    4(k-1)+(3k-5)q
	    +\sqrt{q\bigl(8(k-1)(q^2+k-2)-q(7k^2-10k-1)\bigr)}
	  }{2k(k-1)}
	  \right\rfloor .
	 \end{equation}
	\end{theorem}
	\begin{proof}
	 Let $B$ be the point--line incidence matrix of $\PG(2,q)$ and write
	 \[
	  A=
	  \begin{bmatrix}
	   0&B\\
	   B^{\mathsf T}&0
	  \end{bmatrix}.
	 \]
	 Since
	 \[
	  BB^{\mathsf T}=B^{\mathsf T}B=qI+J,
	 \]
	 the eigenvalues of $A$ are
	 \[
	  q+1,\quad \sqrt q,\quad -\sqrt q,\quad -(q+1)
	 \]
	 with multiplicities $1$, $q^2+q$, $q^2+q$, and $1$, respectively. Hence, in Lemma \ref{EML}, the size of each bipartite part is
	 \[
	  N=q^2+q+1,
	 \]
	 the degree is $d=q+1$, and the largest nontrivial absolute eigenvalue is $\lambda_2=\sqrt q$.

	 For $Y_1$ take the points of the $k$-arcs in a $k$-uniform local arc $\cS$,
	 while for $Y_2$ take the union of secants of the $k$-arcs in $\cS$.
	 Note that a $k$-arc has $\binom{k}{2}$ secants.
	 We find $|Y_1| = k \cdot |\cS|$, $|Y_2| = \binom{k}{2} \cdot |\cS|$,
	 and $E(Y_1, Y_2) = k(k-1) \cdot |\cS|$. Putting $y = |Y_1|$,
	 we can write $y_1 = y$, $y_2 = \frac{k-1}{2} y$, and $E(Y_1, Y_2) = (k-1) y$.
	 Thus,
	 \[
	  | (q^2+q+1) (k-1) y - (q+1) \tfrac{k-1}{2} y^2 | \leq y \sqrt{q \tfrac{k-1}{2} (N-y)(N-\tfrac{k-1}{2}y)}.
	 \]
	 Solving this quadratic equation yields
	 \[
	  y \leq \frac{4(k-1)+(3k-5)q + \sqrt{q [ 8(k-1)(q^2+k-2) - q(7k^2-10k-1)  ]}}{2(k-1)}.
	 \]
	 Since $y=|Y_1|=k|\cS|$, division by $k$ gives exactly the bound in \eqref{eq:ourbnd}.
	\end{proof}

	By \eqref{eq:ourbnd}, we obtain the upper bound of Theorem \ref{thm:main}.

	The bound in Equation \eqref{eq:ourbnd} has magnitude
	\[
		\left(\frac{\sqrt{2}}{k\sqrt{k-1}}+o(1)\right)q^{3/2}
	\]
	for fixed $k$.
	Asymptotically, this matches the earlier bound in Equation \eqref{eq:theirbnd} for $k=4$, but a detailed analysis shows that the bound in \eqref{eq:ourbnd}
	is strictly smaller unless $q \in \{ 2, 4, 5, 7, 11, 16 \}$. Going by the computational results in Section \ref{sec:comp},
	for $k \leq q$ the bound appears to be tight for precisely the following parameters: $(k, q) \in \{ (2, 2), (2, 4), (3, 4), (3, 5) \}$.
	
	In the final part of this section, we introduce another geometric object from local arcs which can be used to provide restrictions on the structure of local arcs.
	\begin{definition}\label{de:quasiarc}
		Let $\Psi$ be a set of points in $\PG(2,q)$. If for any point $P\in \Psi$, there exist at least $t$ tangents through $P$, then we call $\Psi$ a \emph{$t$-quasiarc}.
	\end{definition}
	If we replace ``at least" in Definition \ref{de:quasiarc} by ``exactly", then $\Psi$ is a \emph{semiarc} which is introduced in \cite{CK2012} as a generalization of semiovals and of arcs.
	
	Suppose that $\cS$ is a local arc in $\PG(2,q)$ with $|S|\geq 2$ for any $S\in \cS$. Let $\cL$ be its associated collection of secant line sets. For each $L\in \cL$, we pick up a line $\ell_L$ in $L$ and define $\Phi= \{\ell_L:L\in \cL\}$. Then for any $\ell\in \Phi$, there exist at least two points on $\ell$ not incident with any other line in $\Phi$ which means the dual of $\Phi$ is a $2$-quasiarc.
	\begin{proposition}\label{prop:quasiarc}
		Let $\Psi$ be a $t$-quasiarc in $\PG(2,q)$. If there exists a line $\ell$ such that $|\ell\cap \Psi|=k$, then
		\begin{equation}\label{eq:Psi_bound}
			|\Psi| \leq k+\frac{(q+1-k)q}{t}.
		\end{equation}
	\end{proposition}
	\begin{proof}
		For any $Q\in \Psi\setminus \ell$, there exists at least $t$ tangent lines of $\Psi$ through $Q$. Thus, there are at least $t|\Psi \setminus \ell|$ tangent lines each of which meet $\ell$. Consequently,
		\[
			t(|\Psi|-k) = t|\Psi\setminus \ell|\leq (q+1-k)q,
		\]
		from which \eqref{eq:Psi_bound} is derived.
	\end{proof}
	
	By Proposition \ref{prop:quasiarc} and the relation between local arcs and $2$-quasiarc, we obtain the following restriction on large local arcs.
	\begin{corollary}
		Let $\cS$ be a local arc in $\PG(2,q)$ with $|S|\geq 2$ for every $S\in\cS$, and let $\cL$ be its associated collection of secant line sets. For each $L\in\cL$, choose one line $\ell_L\in L$ and put
		\[
			\Phi=\{\ell_L:L\in\cL\}.
		\]
		For $P\in\PG(2,q)$, set
		\[
			a_P=\bigl|\{\ell\in\Phi:P\in\ell\}\bigr|.
		\]
		If $q>2$, then
		\begin{equation}\label{eq:quasiarc-cor}
			q+1-a_P
			\geq
			\frac{2\bigl(|\cS|-q-1\bigr)}{q-2}.
		\end{equation}
	\end{corollary}
	\begin{proof}
		In the dual plane, the lines in $\Phi$ correspond to the points of a $2$-quasiarc $\Psi$ with
		\[
			|\Psi|=|\Phi|=|\cS|.
		\]
		The dual line corresponding to $P$ meets $\Psi$ in exactly $a_P$ points. Applying Proposition \ref{prop:quasiarc} with $t=2$ and $k=a_P$ gives
		\[
			|\cS|
			\leq
			a_P+\frac{q(q+1-a_P)}{2}.
		\]
		Rearranging this inequality yields \eqref{eq:quasiarc-cor}.
	\end{proof}

	\section{Constructions}\label{sec:cons}
	As we have mentioned in Section \ref{sec:intro}, local arcs of size $\Omega(q)$ in $\PG(2,q)$ can be easily constructed via a partition of ovals. In this section, we focus on the construction of local arcs of size $\Omega(q^{1+\epsilon})$ with $\epsilon>0$.
	 
	Throughout this section, $q$ is odd. We use the following standard coordinatization of $\PG(2,q)$ associated with the planar function $f(x)=x^2$; see, for example, \cite{Deombowski,DembowskiOstrom}. It is obtained from the usual affine coordinatization by the change of coordinates
	\[
		(X,Y)\longmapsto (X,Y+X^2).
	\]
	Since $2$ is invertible in $\F_q$, the parameters $a,b\in\F_q$ below run through all non-vertical affine lines. This also explains why this description is restricted to odd characteristic.
	We use the following point, line, and incidence description:
	
	\begin{tabular}{ll}
		\textbf{Points}: &  $\{(x,y): x,y\in \F_q\} \cup \{(z) : z\in \F_q \cup \{\infty\}\}$;\\
		\textbf{Lines}:  &  $\{[a,b]: a,b\in \F_q\}\cup \{[c] : c\in \F_q \cup \{\infty\}\}$;\\
		\textbf{Incidence}: & $(x,y)$ is incident with $[a,b]$ if $y-b=(x-a)^2$;\\
		& $(x,y)$ is incident with $[c]$ if $x=c$; \\
		& $(z)$ is incident with $[a,b]$ if $z=a$; \\
		& $(z)$ is incident with $[\infty]$;\\
		& $(\infty)$ is incident with $[c]$.
	\end{tabular}

	\subsection{Constructions for $q$ prime}
	First we concentrate on the case with $q=p$ a prime. To construct a large local arc, we need to consider a slightly stronger object over integers as a lifting of local arc in $\PG(2,p)$. Hence, we have to introduce a series of extra notations over integers.

	Informally, we start with a fixed generic local arc and the collections of secant lines sets of its members, place their coordinates in the high base-$205$ digits, and then add low-digit translations from a set $T$. The family $\widetilde{\mathcal S}$ consists of all translated lifted arcs, while $\widetilde{\mathcal L}$ records the correspondingly translated secant-line sets. The restrictions imposed on $T$ are designed to ensure that distinct translated arcs are disjoint and that a translated secant can meet only the translated arc from which it originates.
	
	For $n\in \Z^+$, set $I_n=\{0,1,\dots,n-1\}$.
	For $\mathcal{S}\subseteq \binom{I_n^2}{k}$ and $\mathcal{L}\subseteq \binom{I_n^2}{\binom{k}{2}}$,  we use $(x,y)$ to denote elements in $\bigcup_{S\in \mathcal{S}} S$ and $[a,b]$ for elements in $\bigcup_{L\in \mathcal{L}} L$.
		
	Let $r$ be a prime larger than $n$. Define $\overline{(\cdot)}: \Z \rightarrow \F_r$ by $\overline{x}\equiv x \pmod{r}$. Moreover, for any $x\in \Z^n$, extend the definition of $\overline{(\cdot)}$ to $x$ coordinatewise and also to subsets of $\Z^n$.	
	We further set $\overline{\mathcal{S}}:= \{\overline{S}: S\in \mathcal{S}\}$ and $\overline{\mathcal{L}}:= \{\overline{L}: L\in \mathcal{L}\}$ for ${\cS} \subseteq \binom{I_r^2}{k}$ and $ \cL \subseteq \binom{I_n^2}{\binom{k}{2}}$.
	Suppose that $\mathcal{S}=\{S_1,S_2,\dots,S_m\}$ and $\mathcal{L}=\{L_1,L_2,\dots, L_m\}$ satisfy the following conditions:
	\begin{enumerate}[label=(\alph*)]
		\item $\overline{\mathcal{S}}$ is a $k$-uniform local arc in $\PG(2,r)$ defined via the planar function $x^2$;
		\item for each $j$, $\overline{L_j}$ is the set of secants lines to the $k$-arc $\overline{S_j}$;
		\item for any $(x,y)\in \bigcup_{S\in \mathcal{S}} S$ and any $[a,b]\in \bigcup_{L\in \mathcal{L}} L$, if $(\bar{x}, \bar{y})$ is on $[\bar{a}, \bar{b}]$ then $(x-a)^2=y-b\geq 0$.
	\end{enumerate}
	
	In other words, $\mathcal{S}$ can be viewed as a lifting of the $k$-uniform local arc $\overline{\mathcal{S}}$. However, in general, it is not possible to lift any $k$-uniform local arc to satisfy Condition (c).
	
	When conditions (a), (b) and (c) are all satisfied, for any prime $r'$ with $r'\geq r$ and
	\[
		r'> \max\left\{(x-a)^2 -(y-b) : (x,y)\in \bigcup_{S\in \mathcal{S}} S, [a,b]\in \bigcup_{L\in \mathcal{L}} L\right\},
	\]
	$\mathcal{S}$ always defines a $k$-uniform local arc in $\PG(2,r')$. Thus, we call $\mathcal{S}$ a \emph{generic $k$-uniform local arc}.  In such a case, it is clear that $\mathcal{L}$ is completely determined by $\mathcal{S}$. For the elements $S$ in $\mathcal{S}$, we called $S$ a \emph{generic $k$-arc}.
	
	%By (a), distinct members of $\mathcal{S}$ (resp. $\mathcal{L}$) share no common points (resp. lines). We set 
	%\[\tilde{\mathcal{S}} = \bigcup_{S\in \mathcal{S}} S, \quad \tilde{\mathcal{L}}=\bigcup_{L\in \mathcal{L}} L.\]
	
	\begin{example}
		One can easily verify that the following examples all satisfy conditions (a), (b) and (c).
		\begin{enumerate}[label=(\roman*)]
			\item For $k=2$ and $r=5$, we may take 
			\[
			\mathcal{S} = \{\{(0,4), (4,4)\}, \{(0,3), (2,3)\}, \{(1,3), (3,3)\}\}
			\]
			and 
			\[
			\mathcal{L} = \left\{\{[2,0]\},  \{[1,2]\}, \{[2,2]\}\right\}.
			\]
			\item For $k=3$ and $r=13$, take
			\[
			\mathcal{S} = \{\{(6,12), (2,4), (3,9)\}\}
			\]
			and 
			\[
			\mathcal{L} =\{\{[0,0], [4,8], [3,3]\}\}.
			\]
			It is clear that $\overline{\mathcal{S}}$ actually consists of only one $3$-arc; see Fig.~\ref{fig:ex_local_arc}.

		\tikzset{every picture/.style={line width=0.75pt}} %set default line width to 0.75pt        
		\begin{figure}[http]
			\centering
		\begin{tikzpicture}[
			dot/.style={circle, fill=black, inner sep=2.5pt},
			yscale=1.5, xscale=2.5 % Adjust these to change width/height easily
			]
			
			% Place the dots (Nodes)
			\node[dot, label=above:{$(6,12)$}] (A) at (0,1) {};
			\node[dot, label=above:{$(2,4)$}]  (B) at (1,1) {};
			\node[dot, label=above:{$(3,9)$}]  (C) at (2,1) {};
			
			\node[dot, label=below:{$[0,0]$}]  (D) at (0,0) {};
			\node[dot, label=below:{$[4,8]$}]  (E) at (1,0) {};
			\node[dot, label=below:{$[3,3]$}]  (F) at (2,0) {};
			
			% Draw the lines (Bipartite connections)
			\draw (A) -- (E); \draw (A) -- (F);
			\draw (B) -- (D); \draw (B) -- (F);
			\draw (C) -- (D); \draw (C) -- (E);
			
		\end{tikzpicture}
		\caption{A generic $3$-arc.}
		\label{fig:ex_local_arc}
		\end{figure}
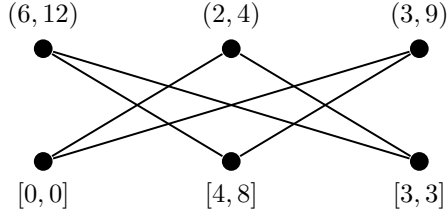
		\end{enumerate}
	\end{example}
	
	The next result shows that generic $k$-arcs always exist for any $k$.
	\begin{lemma}\label{le:generic_arc}
		For any positive integer $k$, define
		\[
			D := \left\{(x,2x): x=\left\lceil \frac{k^2}{2}\right\rceil+2i\text{ with } i=0,1,\dots,k-1 \right\}.
		\]
		Then $D$ is a generic $k$-arc.
	\end{lemma}
	\begin{proof}
		As $2X-b=(X-a)^2$ is a polynomial in $X$ of degree $2$, $\mathcal{S}=\{D\}$ satisfies Condition (a) for any prime larger than $2k-2$. Let $r$ be any prime with $r>k^2+4k-7$. We only have to determine the set $L$ of the lifting of secant lines of $\overline{D}$ and show that Condition (c) holds.
%		$r>\max\{\lceil \frac{k^2}{2}\rceil + 2(k-1), k^2+4k-7\}$
		
		Suppose that $[a,b]$ defines a secant line passing through the two points corresponding to $(x_0,2x_0)$ and $(x_1,2x_1)$, i.e.,
		\[
		\begin{cases}
			2x_0 -b =(x_0-a)^2,\\
			2x_1 -b =(x_1-a)^2,
		\end{cases}
		\]
		for some distinct $x_0,x_1\in \left\{\left\lceil {k^2}/{2}\right\rceil+2i: i=0,\dots, k-1\right\}$.
		Then $a=\frac{x_0+x_1}{2}-1$ and $b=2x_0 - (\frac{x_0-x_1}{2}+1)^2$. It can be readily verified that 
		$$0<a\leq \lceil \tfrac{k^2}{2}\rceil + 2k-4<r$$
		and 
		$$0\leq 2\lceil \tfrac{k^2}{2}\rceil-k^2\leq b\leq 2\left(\lceil \tfrac{k^2}{2}\rceil+2(k-1)\right)-2^2<r.$$
		Hence, $a,b\in \{0,\dots, r-1\}$ and the lifting of all secant lines to $\overline{D}$ satisfy Condition (c).
	\end{proof}
	
	A set $A\subseteq \Z$ is \emph{square-difference-free} if for any distinct $a,b\in A$, the difference $a-b$ is not a square in $\Z$. For a positive integer $N$, a set $A\subseteq [N]$ is \emph{square-difference free mod $N$} if there do not exists different $x,y\in A$ such that $x-y$ is a square $\mathrm{mod}~N$. It is obvious that a set which is square-difference-free mod $N$ must be square-difference-free.
	
	In \cite{Ruzsa1984}, Ruzsa proved that for any positive integer $N$, there exist square-differ\-ence-free sets in $[N]$ with size $\Omega(N^{0.733077})$. This result was slightly improved in \cite{BG2008} as follows.
	\begin{lemma}\label{le:square-diff-free}
		For any positive integer $N$, there exists a set in $[N]$ which is square-difference-free with size $\Omega(N^{0.7334})$.
	\end{lemma}
	
	In \cite{BG2008}, the following result is used for the construction of square-difference-free sets mod $m^{2k}$ with $m=205$ from which Lemma \ref{le:m=205} is derived.
	\begin{lemma}\label{le:m=205}
		For $m=205$, the set $A = \{0, 2, 8, 14, 77, 79, 85, 96, 103, 109, 111, 181\}$ is  square-difference-free modulo $m$.
	\end{lemma}
	
	For any $k\geq 2$ and any prime $r$, suppose that there exist $\mathcal{S}$ and $\mathcal{L}$ satisfying conditions (a), (b) and (c).
	For a prime $p$, choose the largest positive even integer $t$ satisfying $(r^2+3r+1)m^t\leq p$ with $m=205$, i.e.,
	\begin{equation}\label{eq:condition_t_general}
		t < \log_{m}p - \log_{m}(r^2+3r+1).
	\end{equation}
	As we require that $t\geq 2$, $p$ has to be larger than $m^2(r^2+3r+1)$.
	For any $(x,y)\in I_r^2$, define its \emph{lifting} $\lambda_t(x,y)=(x\cdot m^{t/2}, y\cdot m^t)\in [0,p-1]^2$.

	Put $M=m^{t/2}$. We further define
	\begin{equation}\label{eq:T_translation}
		T=
		\left\{
		\left(u,\sum_{i=0}^{t-1}y_i m^i\right):
		u\in I_M,\ 
		y_{\mathrm{even}}\in A,\ 
		y_{\mathrm{odd}}\in\{0,\dots,m-1\}
		\right\}.
	\end{equation}
	Here $A$ is defined in Lemma \ref{le:m=205}. Since $|A|=12$,
	\[
		|T|
		=M\,|A|^{t/2}m^{t/2}
		=(m\sqrt{|A|})^t.
	\]
	
	By the definition of $t$, 
	\begin{equation}\label{eq:|T|}
			|T|=\Omega(p^{1+\log_m\sqrt{|A|}}),
	\end{equation}
	as $p\rightarrow \infty$ through odd primes.
	
	We extend the definition $\lambda_t$ in the natural way to subsets $S$ of $I_r^2$, i.e., 
	\[
	\lambda_t(S) =\{\lambda_t(x,y): (x,y)\in S\}.
	\]
	For any subset $A\subseteq \Z^2$ and any $\tau=(u,v)\in T$, set
	\[
	A+\tau \pmod{p} =\{(a_1+u \pmod{p}, a_2+v \pmod{p}):(a_1,a_2)\in A\}\subseteq \F_p^2.
	\]
	Define
	\begin{equation}\label{eq:lifting_Omega}
		\tilde{\mathcal{S}} := \{\lambda_t(S)+\tau \pmod{p}: S\in \mathcal{S}, \tau\in T\}
	\end{equation}
	and 
	\[
	\tilde{\mathcal{L}} := \{\lambda_t(L)+\tau \pmod{p}: L\in \mathcal{L},  \tau\in T\}.	
	\]
	Obviously, $\lambda_t(L)+\tau \pmod{p}$ is the set of secant lines to the arc $\lambda_t(S)+\tau \pmod{p}$.
	
	\begin{theorem}\label{th:construction_1}
		For any $k\geq 2$ and any prime $r$, suppose that $\mathcal{S}$ and $\mathcal{L}$ satisfy conditions (a), (b) and (c). Let $m$ and $A$ be defined as in Lemma \ref{le:m=205}.
		
		For any prime $p>m^2(r^2+3r+1)$, let $t$, $T$, and $\tilde{\mathcal{S}}$ be defined as in \eqref{eq:condition_t_general}, \eqref{eq:T_translation} and \eqref{eq:lifting_Omega}, respectively.
		Then $\tilde{\mathcal{S}}$ is a $k$-uniform local arc and 
		\[|\tilde{\mathcal{S}}|=\Omega( p^{1+\log_{m}\sqrt{|A|}})=\Omega(p^{1.2334}),\] 
		as $p\rightarrow \infty$ through all odd primes.
	\end{theorem}
	\begin{proof}
		The cardinality of $\tilde{\mathcal{S}}$ follows directly from $|\tilde{\mathcal{S}}|=|T|\cdot |\mathcal{S}|$ and \eqref{eq:|T|}.
		
		We divide the proof into four main steps.
		
		\noindent\emph{Disjointness.}
		We first check that distinct translated members are disjoint. Suppose that
		\[
			\lambda_t(x_0,y_0)+(u,v)
			\equiv
			\lambda_t(x_0',y_0')+(u',v')
			\pmod p,
		\]
		where $(u,v),(u',v')\in T$. By the choice of $t$, all displayed integer coordinates lie in $I_p$, so the congruence is an equality in $\Z^2$. In the first coordinate, $u,u'\in I_M$, and hence uniqueness of quotient and remainder modulo $M=m^{t/2}$ gives $x_0=x_0'$ and $u=u'$. In the second coordinate, $0\leq v,v'<m^t$, so the same argument gives $y_0=y_0'$ and $v=v'$. The original members of $\mathcal S$ are disjoint, and therefore so are the translated members.

		It remains to prove that every line in the translated secant family contains only its two intended points. Let
		\[
			\cP=\bigcup_{S\in\tilde{\mathcal S}}S,
			\qquad
			\cL=\bigcup_{L\in\tilde{\mathcal L}}L.
		\]
		\noindent\emph{Lifting from $\F_p$ to $\Z$.}

		By definition, the fixed line $\ell$ can be represented by $[a,b]\in \Z^2$ with $a=\sum_{i=0}^{t/2-1} a_i m^i+a_{t/2}m^{t/2}$ and $b=\sum_{i=0}^{t-1} b_i m^i+b_t m^t$, where $(\sum_{i=0}^{t/2-1} a_i m^i,\sum_{i=0}^{t-1} b_i m^i) \allowbreak \in T$ and $[a_{t/2}, b_t]$ belongs to a set of secant lines $L\in \mathcal{L}$.
		
		Suppose that a point $P\in \cP$ is on $\ell$ which means
		\begin{equation}\label{eq:(x,y)_on_[a,b]}
			(x-a)^2\equiv y-b \pmod{p}.
		\end{equation}
		By definition, $P$ can be represented by $(x,y)\in \Z^2$ with $x=\sum_{i=0}^{t/2-1} x_i m^i+x_{t/2}m^{t/2}$ and $y=\sum_{i=0}^{t-1} y_i m^i+y_t m^t$, where $(\sum_{i=0}^{t/2-1} x_i m^i,\sum_{i=0}^{t-1} y_i m^i)\in T$ and $(x_{t/2}, y_t)$ belongs to a $k$-arc $S\in \mathcal{S}$.
		
		By \eqref{eq:condition_t_general}, it is clear that
		\begin{equation}\label{eq:(x-a)^2_range}
			0\leq (x-a)^2 \leq \left(2(m-1)\sum_{i=0}^{t/2-1}m^i + (r-1) m^{t/2}\right)^2 <(r+1)^2 m^t\leq p-rm^t.
		\end{equation}
		Moreover, by the definition of $y$ and $b$,
		\[
		y-b>-(r-1)m^t -\sum_{i=0}^{t-1}(m-1)\cdot m^i =-r m^t+1,
		\]
		which means
		\begin{equation}\label{eq:y-b_range}
			y-b\geq -rm^t.
		\end{equation}
		By \eqref{eq:(x-a)^2_range} and \eqref{eq:y-b_range}, \eqref{eq:(x,y)_on_[a,b]} implies
		\begin{equation}\label{eq:(x,y)_on_[a,b]_strong}
			(x-a)^2= y-b\geq 0.
		\end{equation}
		
		\medskip
		\noindent\emph{Digit comparison.} We next show that
		\begin{equation}\label{eq:same_translation}
			\left(\sum_{i=0}^{t/2-1} x_i m^i,\sum_{i=0}^{t-1} y_i m^i\right)=\left(\sum_{i=0}^{t/2-1} a_i m^i,\sum_{i=0}^{t-1} b_i m^i\right).
		\end{equation}
		
		If $y=b$, then \eqref{eq:(x,y)_on_[a,b]_strong} gives $x=a$, and uniqueness of the base-$m$ expansions yields \eqref{eq:same_translation}. Assume now that $y\neq b$, and let
		\[
			i_0=\min\{i:0\leq i\leq t,\ y_i\neq b_i\}.
		\]
		If $i_0<t$, then $m^{i_0}$ is the exact power of $m$ dividing $y-b$. Since $m=5\cdot41$ is squarefree, the exact power of $m$ dividing a square is even: it is the minimum of the $5$-adic and $41$-adic valuations, both of which are even. Equation \eqref{eq:(x,y)_on_[a,b]_strong} therefore implies that $i_0$ is even and that $m^{i_0/2}\mid x-a$. Dividing by $m^{i_0}$ and reducing modulo $m$ gives
		\[
			y_{i_0}-b_{i_0}
			\equiv
			\left(\frac{x-a}{m^{i_0/2}}\right)^2
			\pmod m.
		\]
		Because $i_0$ is even, both $y_{i_0}$ and $b_{i_0}$ belong to $A$, contradicting Lemma \ref{le:m=205}. Hence $i_0=t$, so the two vertical translations agree and
		\[
			y-b=(y_t-b_t)m^t.
		\]

		Now $m^t=M^2$ divides $(x-a)^2$. Comparing the $5$-adic and $41$-adic valuations gives $M\mid x-a$. Write
		\[
			u_x=\sum_{i=0}^{t/2-1}x_i m^i,
			\qquad
			u_a=\sum_{i=0}^{t/2-1}a_i m^i.
		\]
		Then $u_x,u_a\in I_M$ and
		\[
			x-a=(x_{t/2}-a_{t/2})M+(u_x-u_a).
		\]
		Thus $M\mid u_x-u_a$. Since $|u_x-u_a|<M$, we obtain $u_x=u_a$, proving \eqref{eq:same_translation}.
		\medskip
		
		\noindent\emph{Uniqueness of the original secant.}
		Once the translations agree,
		\[
			x-a=M(x_{t/2}-a_{t/2}),
			\qquad
			y-b=M^2(y_t-b_t).
		\]
		Therefore \eqref{eq:(x,y)_on_[a,b]_strong} reduces to
		\[
			(x_{t/2}-a_{t/2})^2=y_t-b_t.
		\]
		By conditions (a), (b), and (c), the base point $(x_{t/2},y_t)$ is one of the two points of the unique base $k$-arc lying on the secant $[a_{t/2},b_t]$. Hence the translated line $\ell=[a,b]$ contains exactly its two intended points and no point of another translated member.
	\end{proof}

	\begin{remark}
		For a given $k\geq 2$, to guarantee the existence of a generic $k$-uniform local arc $\mathcal{S}$, the prime $r$ has to be large enough. Lemma \ref{le:generic_arc} provides a way to find such $\mathcal{S}$. 
		In fact, the proof of Theorem \ref{th:construction_1} shows that $\{\lambda_t(S)+\tau: S\in \mathcal{S}, \tau\in T\}$ is also a generic $k$-uniform local arc. 
	\end{remark}

	\subsection{Constructions for $q$ prime power}
	In the second part of this section, we show that the construction of local arcs in $\PG(2,p)$ with prime $p$ can be lifted to $\PG(2,q)$ where $q$ is a power of $p$. Its proof is similar to Proposition 4.1 in \cite{HPVZ2026}.
\begin{theorem}\label{th:construction_q}
	Let $m$ and $k$ be integers with $m>1$ and $k\geq 2$. 
		As $p\rightarrow \infty$ through odd primes and $q=p^m$, there exists a $k$-uniform local arc $\tilde{\mathcal{S}}$ in $\PG(2,q)$ with
	\[
		|\tilde{\mathcal{S}}|=\Omega(q^{d_m}),
	\]
	where $d_m$ is given in \eqref{eq:localarc}.
\end{theorem}
\begin{proof}
	For $p$ large enough, by Theorem \ref{th:construction_1}, there exists $k$-uniform local arc $\mathcal{S}$ in $\PG(2,p)$.
	Let $\mathcal{L}$ denote the set of secants lines to $\mathcal{S}$, i.e.
	\[
	\mathcal{L} = \left\{ \left\{[a,b]\in \F_p^2: [a,b] \text{ is a secant line of }S\right\}: S\in \mathcal{S}\right\}.
	\]	
	We separate the proof into 4 cases depending on the value of $m$. In each case, we embed $\mathcal{S}$ and $\mathcal{L}$ in $\PG(2,q)$ and construct a set $T$ of translations acting them to obtain $\tilde{\mathcal{S}}$ and $\tilde{\mathcal{L}}$. Our goal is to show that for any line $\tilde{\ell}:=\ell+\tau$ with $\ell$ in some $L\in \mathcal{L}$ and $\tau\in T$, if a point $\tilde{P}$ in $\bigcup_{S\in \tilde{\mathcal{S}}}$ is on $\tilde{\ell}$, then $\tilde{P}$ must belong to $S+\tau$ where $S\in \mathcal{S}$ takes $L$ as its associated set of secant lines. Together with the condition that $\mathcal{S}$ is a local arc,  it follows that $\tilde{\ell}$ meets exactly one element in $\tilde{\mathcal{S}}$, that is $S+\tau$, in two points and $\tilde{\ell}$ is not incident with any other element in $\tilde{\mathcal{S}}$.
	
	\textbf{Case I: $m=2$.} Let $\alpha$ be an element in $\F_{p^2}$ such that $\F_p[\alpha] =\F_{p^2}$ which means the minimal polynomial of $\alpha$ over $\F_p$ is of degree $2$. Set
	\[
		T := \{(0,g_1\alpha ) : g_1\in \F_p\}.
	\]
	Letting $T$ act on $\mathcal{S}$ and $\mathcal{L}$ by addition, we get
	\[
		\tilde{\mathcal{S}} = \left\{
			\left\{ (f_0, g_0+g_1\alpha): (f_0, g_0)\in S
			\right\}: g_1\in \F_p, S\in \mathcal{S}
		\right\}
	\] 
	and
	\[
		\tilde{\mathcal{L}} = \left\{
			\left\{
				[a_0,b_0+b_1\alpha]: (a_0,b_0)\in L
			\right\}: b_1\in \F_p, L\in \mathcal{L}
		\right\}.
	\]
	
	It is clear that $|\tilde{\mathcal{S}}|=|\mathcal{S}| \cdot p$. 	By Theorem \ref{th:construction_1}, we can construct $\tilde{\mathcal{S}}$ of size
	\[
	|\tilde{\mathcal{S}}| =\Omega (q^{1.1167}),
	\]
	as $p\rightarrow \infty$ through odd primes.
	
	For a given line $\tilde{\ell}=[a_0,b_0]+\tau$ where $\tau=(0,b_1\alpha)\in T$, suppose that $(f_0,g_0)+(0,g_1\alpha)$ is a point, denoted by $\tilde{P}$, incident with $\tilde{\ell}$, that is,
	\[
		(f_0-a_0)^2=g_0-b_0+(g_1-b_1)\alpha.
	\]
	By the definition of $\alpha$, we must have $g_1=b_1$ which means $\tau=(0,b_1\alpha)=(0,g_1\alpha)$ and the point $\tilde{P}=(f_0,g_0)+\tau$.
	
	\textbf{Case II: $m=2t$ with $t>1$.}
	Put $F=\F_{p^2}$. Choose $\alpha\in\F_{p^{2t}}$ such that
	\[
		F[\alpha]=\F_{p^{2t}},
	\]
	so the minimal polynomial of $\alpha$ over $F$ has degree $t$. Independently, choose a fixed nonsquare $\eta\in F$. Set
	\[
		s=\left\lfloor\frac{t+1}{2}\right\rfloor
	\]
	and define
	\[
	T=\left\{
	(f(\alpha),g(\alpha)):
	f=\sum_{i=1}^{s-1}f_iX^i,\ f_i\in F,\ 
	g=\sum_{i=1}^{t-1}g_iX^i,\ 
	g_{\mathrm{even}}\in\eta\F_p,\ 
	g_{\mathrm{odd}}\in F
	\right\}.
	\]
	Every nonzero element of $\F_p$ is a square in $F$, since for $c\in\F_p^*$,
	\[
		c^{(p^2-1)/2}=(c^{p-1})^{(p+1)/2}=1.
	\]
	Consequently, every nonzero element of $\eta\F_p$ is a nonsquare in $F$, and the nonzero difference of two distinct elements of $\eta\F_p$ is again a nonsquare.

	Let $\mathcal{S}'$ be a $k$-uniform local arc constructed in Case I and $\mathcal{L}'$ its associated sets of secant lines. Set $\tilde{\mathcal{S}} =\{\mathcal{S}' + \tau: \tau\in T\}$ and $\tilde{\mathcal{L}} =\{\mathcal{L}' + \tau: \tau\in T\}$. Clearly,
	\[
		|\tilde{\mathcal{S}}| = 
		\begin{cases}
			|\mathcal{S}'|\cdot p^{5(s-1)}, & 2\nmid t;\\
			|\mathcal{S}'|\cdot p^{5s-3}, & 2\mid t,
		\end{cases}
	\]
	which further equals to
	\[
	|\tilde{\mathcal{S}}| = 
	\begin{cases}
		|\mathcal{S}'|\cdot p^{5(t-1)/2}, & 2\nmid t;\\
		|\mathcal{S}'|\cdot p^{5t/2-3}, & 2\mid t.
	\end{cases}
	\]
	By Case I, we can construct $\tilde{\mathcal{S}}$ of size
	\[
	|\tilde{\mathcal{S}}| =
	\begin{cases}
		\Omega(q^{1.25-0.2666/m}), & 2\nmid t;\\
		\Omega(q^{1.25-0.7666/m}), & 2\mid t,
	\end{cases}
	\]
	as $p\rightarrow \infty$ through all odd primes.
	
	Take any secant line 
	$\tilde{\ell} :=[u(\alpha),v(\alpha)]$ with
	\[u=u_{0}+\sum_{i=1}^{s-1}u_i X^i\] 
	and 
	\[v=v_0+\sum_{i=1}^{t-1} v_i X^i\] 
	which means $\tilde{\ell}=[u_0,v_0]+\tau$ with $\tau=(\sum_{i=1}^{s-1}u_i \alpha^i, \sum_{i=1}^{t-1} v_i \alpha^i)\in T$ and $[u_{0}, v_{0}]\in L$ for some $L\in \mathcal{L}'$. 
	
	Suppose that $(f(\alpha),g(\alpha))\in \tilde{\mathcal{S}}$ is a point $\tilde{P}$ on $\tilde{\ell}$, that is
	\begin{equation}\label{eq:(f,g)_on_[u,v]}
		(f(\alpha)-u(\alpha))^2 = g(\alpha) -v(\alpha),
	\end{equation}
	where $\deg(f)\leq s-1$ and $\deg(g)\leq t-1$.
	
	As the degrees of polynomials $(f-u)^2$ and $g-v$ are both smaller than $t$, \eqref{eq:(f,g)_on_[u,v]} holds if and only if
	\begin{equation}\label{eq:poly_condition_for_q}
		(f-u)^2 = g-v
	\end{equation}
	in $\F_{p^2}[X]$.
	
	Assume that $\sum_{i=1}^{t-1}v_i X^i$ and $\sum_{i=1}^{t-1}g_i X^i$ are different, and let
	\[
		i_M=\max\{i:g_i\neq v_i\}.
	\]
	Equation \eqref{eq:poly_condition_for_q} shows that $i_M$ must be even, say $i_M=2j$, because the degree of a nonzero square is even. By construction, $g_{2j},v_{2j}\in\eta\F_p$, so their nonzero difference is a nonsquare in $F=\F_{p^2}$. On the other hand, the leading coefficient of $(f-u)^2$ is a square in $F$, contradicting \eqref{eq:poly_condition_for_q}. Therefore,
	we must have
	\[
	\sum_{i=1}^{t-1}v_i X^i=\sum_{i=1}^{t-1} g_i X^i.
	\]
	Thus $g-v=g_0-v_0$ is constant. Equation \eqref{eq:poly_condition_for_q} then shows that $(f-u)^2$ is constant, and hence $f-u=f_0-u_0$ is constant as well. Therefore,
	\[
		(f(\alpha),g(\alpha))=(f_0,g_0)+\tau.
	\]

	\textbf{Case III: $m=2t+1$.} Let $\alpha$ be an element in $\F_{p^{m}}$ such that $\F_{p}[\alpha] =\F_{p^{m}}$ which means the minimal polynomial of $\alpha$ over $\F_{p}$ is of degree $m$.
	Set $M_1=4$ and $M_2=16$.
	
	Let $A$ be a square-difference-free subset in $[\lfloor p/M_1\rfloor]$. Define $T\subseteq\F_q^2$ as follows:
	\begin{align*}
	T=\left\{
	(f(\alpha), g(\alpha)) :  f=\sum_{i=1}^{t} f_i X^i, f_i\in [\lfloor \sqrt{p/M_2}\rfloor], \right. \\
	\left. g=\sum_{i=1}^{m-1} g_i X^i, g_{\mathrm{even}}\in A, g_{\mathrm{odd}}\in \F_{p}
	\right\}.
	\end{align*}
	Here positive integers $f_i$ and $g_{\mathrm{even}}$ are viewed as elements in $\F_p$ under the modulo operation.
	
	Let $\cS$ be a $k$-uniform local arc in $\PG(2,p)$ as defined in the very beginning of this proof, set $\tilde{\mathcal{S}} :=\{\mathcal{S} + \tau: \tau\in T\}$ and $\tilde{\mathcal{L}} :=\{\mathcal{L} + \tau: \tau\in T\}$. Clearly,
	\[
	|\tilde{\mathcal{S}}| = |\mathcal{S}|\cdot \left(\left\lfloor\sqrt{\frac{p}{M_2}}\right\rfloor\right)^{t} p^{t} |A|^t.
	\]

	As in Case~II, an incidence between differently translated objects gives a polynomial identity
	\begin{equation}\label{eq:poly-odd}
		F(X)^2=G(X)
		\quad\text{in }\F_p[X],
	\end{equation}
	with $\deg F\le t$ and $\deg G\le2t<m$. If the nonconstant part of $G$ is nonzero, let $2j$ be its largest degree. Its leading coefficient is
	\[
	a-a'\equiv (c-c')^2\pmod p,
	\]
	where $a,a'\in A$ and $c,c'\in [\lfloor \sqrt{p/M_2}\rfloor]$. Here
	\[
	\abs{a-a'}<\frac p{M_1},
	\qquad
	0\le(c-c')^2<\frac p{M_2}.
	\]
	Since $M_1=4$ and $M_2=16$, the congruence can only be equality in $\Z$. This contradicts the square-difference-free property of $A$. Hence the nonconstant part of $G$ is zero, and then the nonconstant part of $F$ is zero. Thus the translations coincide.
	
	The translation set has size
	\[
	\abs{T}=\left( \left\lfloor \sqrt{\frac{p}{M_2}} \right\rfloor \right)^t p^t\abs A^t
	=\Omega\bigl(p^{2.2334\cdot t}\bigr).
	\]
	Combining this with $|\cS|=\Omega(p^{1.2334})$ and using $m=2t+1$ gives
	\[
	\abs{\widetilde\cS}
	=\Omega\bigl(p^{1.2334+2.2334\cdot t}\bigr)
	=\Omega\bigl(q^{1.1167+0.1167/m}\bigr).
	\]
		
	\textbf{Case IV: $m=2t$ with $t>1$.} We can use the construction in Case III again and omit its proof. The cardinality of $\tilde{\mathcal{S}}$ becomes
	\[
		|\tilde{\mathcal{S}}|= \Omega(q^{1.1167}).
	\]
	The construction in Case II is always better than the one above for $t$ odd and for $t>2$ if $t$ is even. Hence it only works better for $m=4$.
	
	In summary, by combining Cases I–IV, we finish the proof.
\end{proof}

	\section{Small Numbers}\label{sec:comp}
	For small $k$ and $q$ a fairly standard integer linear program (ILP) can determine the size of a maximum $k$-uniform local arcs (see below).

	\begin{center}
	\begin{tabular}{c|rrrrrrrrrrr}
	 $q \backslash k$ & 2 & 3 & 4 & 5 & 6 & 7 & 8 & 9 & 10 & 11 & 12\\ \hline
	 2 & 3 & 1 & 1 \\
	 3 & 4 & 1 & 1 \\
	 4 & 7 & 4 & 1 & 1 & 1 \\
	 5 & 9 & 5 & 1 & 1 & 1 \\
	 7 & 13& 8 & 3 & 1 & 1 & 1 & 1 \\
	 8 &$\geq17$&9&4&3& 1& 1 & 1 & 1 & 1 \\
	 9 &$\geq20$&$\geq9$&4&2 & 1 & 1 & 1 & 1 & 1 \\
	 11& $\geq24$ & $\geq10$ & 7 & 2 & 2 & 1 & 1 & 1 & 1 & 1 & 1
	\end{tabular}
	\end{center}

	An entry marked by ``$\geq$'' is only a certified lower bound: the ILP found a feasible local arc of the displayed size but did not certify optimality within the one-day computational limit. Unmarked entries are the exact values reported.

	For $q=7,8,9$ and $k=4$, this matches the results in \cite[Example 1]{FTFC2024}.
	For $q=11$ and $k=4$, our table implies that the largest Singleton-optimal
	$(n, K, 6; 3)_{11}$-LRC with disjoint repair groups has length $4 \cdot 7 = 28$.\footnote{%
	We used Gurobi 11.0.2 on an AMD EPYC 9684X within GAP \cite{GAP} with the packages Gurobify, FinInG, and Images.
	No search ran for longer than a day.}

	The ILP used for obtaining these numbers is as follows.
	Let $\cP = \{ P_1, \ldots, \allowbreak P_{q^2+q+1}\}$ denote the points of $\PG(2, q)$,
	and let $\cL = \{ L_1, \ldots, L_{q^2+q+1} \}$ denote the lines of $\PG(2, q)$.
	Theorem \ref{thm:upper} gives an upper bound on the number $m$ of $S_i$'s.
	We take the following $[2(q^2+q+1)+1]m$ binary variables: $P_{i,j}$, $L_{i,j}$, $M_j$
	for all $1 \leq i \leq q^2+q+1$, $1 \leq j \leq m$.
	The sum $\sum_j M_j$ counts the number of $\cS_i$ and we maximize
	\[
	 M_1 + M_2 + \cdots + M_m.
	\]
	The variable $P_{i,j}$ indicates if $P_i$ is in $\cS_j$.
	The variable $L_{i,j}$ indicates if $L_i$ is a secant of $\cS_j$.
	We impose the following constraints:
	\begin{itemize}
	 \item $M_j \geq M_{j+1}$ for all $1 \leq j \leq m-1$, \hfill (symmetry breaking)
	 \item $-k M_j + \sum_i P_{i,j} = 0$ for all $j$, \hfill (each $S_j$ has $k$ points)
	 \item $\sum_{P_{i'} \in L_{i}} P_{i',j} \leq 2$ for all $i,j$, \hfill ($S_j$ is an arc)
	 \item $0 \leq -2L_{i,j} + \sum_{P_{i'} \in L_{i}} P_{i',j} \leq 1$ for all $i,j$, \hfill ($L_{i,j}=1$ iff secant of $S_j$)
	 \item $\sum_j P_{i,j} \leq 1$ for all $i$, \hfill ($S_j \cap S_{j'} = \emptyset$)

	 \item $m L_{i,j} + \sum_{P_{i'} \in L_{i}, j' \neq j} P_{i',j} \leq m$ for all $i,j$, \hfill (secants of $S_j$ avoid $S_{j'}$).
	\end{itemize}
	We added further symmetry breaking by letting $P_{i,1}$ correspond to the largest arc contained in all arcs of size $k$.
	
	\section{Concluding Remarks} \label{sec:fut}
	After this document has been accepted pending minor revisions, using a well-known LLM, Pohoata \cite{Pohoata2026} improved the results
	in \cite{HPVZ2026}. These improvements
	adopt \textit{mutatis mutandis} to our problem of \textit{local arcs}.
	Thus, Theorem \ref{thm:main} can be improved to the following.
		
	\begin{theorem}\label{thm:po}
		Fix an integer $k\geq 2$ and a prime $r \geq 5$.
		Then there exist $k$-uniform local arcs $\cS$ in $\PG(2,p)$ with $$|\cS|=\Omega(p^{1.5 - 2/(r-1)}),$$
		as $p\rightarrow \infty$ through the primes satisfying $p \equiv \pm1 \pmod{r}$.
	\end{theorem}
	It seems plausible that analogous improvements of Theorem \ref{thm:main}
	as in Theorem \ref{thm:po} hold for general prime powers $q$.
	
	We showed in the above that a $k$-uniform local arc can be significantly larger than $O(q)$ for fixed $k$.
	It would be interesting to investigate the possible sizes of maximum $k$-uniform local arcs for other regimes where $k$ is not constant.
	Clearly, $k \leq q+2$ for $q$ even, respectively, $k \leq q+1$ for $q$ odd,
	with equality if and only if the $k$-uniform local arc corresponds to a hyperoval, respectively,
	oval.
	One interesting regime is
	\[
		k=Cq+o(q),
		\qquad 0<C\leq1.
	\]
	Substituting this relation directly into the exact bound in Theorem \ref{thm:upper} gives
	\[
		|\cS|
		\leq
		\frac{3C+\sqrt{C(8-7C)}}{2C^2}+o(1).
	\]
	Thus, in this regime the number of constituent arcs is bounded independently of $q$. Determining the optimal constant remains an interesting problem.

	Call $\cS$ a \emph{$k$-uniform $m$-wise local arc} if the following conditions are satisfied:
	\begin{itemize}
		\item $S_i\cap S_j=\emptyset$ for any distinct $S_i,S_j\in \cS$, and
		\item$\bigcup_{i=1}^m S_i$ is an arc for any $m$ subsets $S_1$, $S_2$, $\dots, S_m\in \cS$.
	\end{itemize}
	It would be interesting to generalize our results to this setting.
	
	For $k\neq 4$, $k$-uniform local arcs fail to induce LRCs under the framework proposed in \cite{FTFC2024}. In Theorem 9 of \cite{FTFC2024}, we find one more geometric structure, living in $PG(3, q)$, which corresponds to a certain type of LRC. This prompts us to ask if more finite geometric structures can correspond to special types of LRCs.

	\bigskip
	\paragraph*{\bf Declaration of competing interest}
	The authors declare that they have no known competing financial interests or personal 
	relationships that could have appeared to influence the work reported in this paper.
	\bigskip
	\paragraph*{\bf Declaration of generative AI use}
	The authors declare that no generative artificial intelligence (AI) or AI-assisted technologies were used in the writing process or preparation of this manuscript. 
	\bigskip
	\paragraph*{\bf Acknowledgement} We thank Jacques Verstra\"ete
	for discussing the construction from \cite{HPVZ2026} with us. Yue Zhou is supported by the National Natural Science Foundation of China (No.\ 12371337) and the Natural Science Foundation of Hunan Province (No.\ 2023RC1003).
	\bigskip
	\paragraph*{\bf Data availability}
	No data was used for the research  described in the article.

\end{document}